\documentclass[reqno, 11pt,a4paper]{amsart}

\usepackage{array}
\usepackage{pdfsync}
\usepackage{stmaryrd} 
\usepackage{marginnote}
\usepackage[colorlinks=true, pdfstartview=FitV, linkcolor=blue,citecolor=blue, urlcolor=blue,pagebackref=false]{hyperref}
\usepackage{color}
 \definecolor{refkey}{gray}{.5}
 \definecolor{labelkey}{gray}{.5}
\definecolor{light}{gray}{.9}
\usepackage[showlabels,sections,floats,textmath,displaymath]{}
\usepackage[percent]{overpic}

\usepackage{tikz}
\usetikzlibrary{arrows}

\usepackage{calc,amsfonts,amsthm,amscd,epsfig,psfrag,amsmath,amssymb,enumerate,paralist,mathrsfs,mathtools,dsfont,graphicx}
\usepackage[american]{babel}
\mathtoolsset{showonlyrefs,showmanualtags}

\usepackage[showlabels,sections,floats,textmath,displaymath]{}

\reversemarginpar
\newlength\fullwidth
\numberwithin{equation}{section}

\DeclareMathSymbol{\leqslant}{\mathalpha}{AMSa}{"36} 
\DeclareMathSymbol{\geqslant}{\mathalpha}{AMSa}{"3E} 
\DeclareMathSymbol{\eset}{\mathalpha}{AMSb}{"3F}     
\renewcommand{\leq}{\;\leqslant\;}                   
\renewcommand{\geq}{\;\geqslant\;}                   

\DeclareSymbolFont{bbold}{U}{bbold}{m}{n}
\DeclareSymbolFontAlphabet{\mathbbold}{bbold}

\renewcommand{\b}{\beta}

\newcommand{\D}{\Delta}
\newcommand{\grad}{\nabla}

\renewcommand{\b}{\beta}
\renewcommand{\l}{\lambda}
\renewcommand{\L}{\Lambda}

\renewcommand{\l}{\lambda}
\renewcommand{\a}{\alpha}
\renewcommand{\d}{\delta}
\renewcommand{\t}{\tau}

\newcommand{\g}{\gamma}

\newcommand{\e}{\varepsilon}

\renewcommand{\o}{\omega}
\renewcommand{\O}{\Omega}

\newcommand{\tc}{\thinspace |\thinspace}

\newtheorem{theorem}{Theorem}[section]
\newtheorem{pseudo-theorem}{Pseudo-Theorem}

\newtheorem{lemma}[theorem]{Lemma}
\newtheorem{proposition}[theorem]{Proposition}

\newtheorem{definition}[theorem]{Definition}

\newtheorem{maintheorem}{Theorem}

\newcommand{\Z}{\mathbb Z}

\newcommand{\cB}{\ensuremath{\mathcal B}}
\newcommand{\cC}{\ensuremath{\mathcal C}}

\newcommand{\cH}{\ensuremath{\mathcal H}}

\newcommand{\cU}{\ensuremath{\mathcal U}}

\newcommand{\cZ}{\ensuremath{\mathcal Z}}

\newcommand{\bbE}{{\ensuremath{\mathbb E}} }

\newcommand{\bbN}{{\ensuremath{\mathbb N}} }

\newcommand{\bbP}{{\ensuremath{\mathbb P}} }

\newcommand{\bbR}{{\ensuremath{\mathbb R}} }

\newcommand{\bbZ}{{\ensuremath{\mathbb Z}} }

\newcommand{\sC}{{\ensuremath{\mathscr C}}}

\newcommand{\wt}{\widetilde }

\newcommand{\ind}{{\bf 1}}

\let\a=\alpha \let\b=\beta   \let\d=\delta  \let\e=\varepsilon
 \let\g=\gamma       \let\l=\lambda
      \let\o=\omega      
   \let\t=\tau   
  
\let\D=\Delta   \let\G=\Gamma  \let\L=\Lambda 
\let\O=\Omega

\title[Entropic repulsion in $|\grad \phi|^p$ surfaces]{ Entropic repulsion in $|\grad \phi|^p$ surfaces: \\ a large deviation bound for all $p\geq 1$}
\author{Pietro Caputo, Fabio Martinelli and Fabio Lucio Toninelli}
\address{Dipartimento di Matematica e Fisica, Universit\`a Roma
  Tre, Largo S. Murialdo 1, 00146 Roma, Italy}
\email{caputo@mat.uniroma3.it, martin@mat.uniroma3.it}

\address{Universit\'e de Lyon, CNRS and Institut Camille Jordan, Universit\'e Lyon 1,
    43 bd
 du 11 novembre 1918, 69622 Villeurbanne, France}
\email{toninelli@math.univ-lyon1.fr}

\begin{document}
\begin{abstract}
  We consider the 
  $(2+1)$-dimensional generalized solid-on-solid (SOS) model, that is the random discrete surface with a gradient potential of the form $|\grad\phi|^{p}$, where
  $p\in [1,+\infty]$.   We show that at low temperature, 
 for a square region $\L$ with side $L$, both
  under the infinite volume measure and under the measure with zero
  boundary conditions around $\L$, the  probability that the surface is nonnegative in $\L$ behaves
  like $\exp(-4\b\tau_{p,\beta} L H_p(L) )$, where $\b$ is the inverse temperature, $\tau_{p,\beta}$ is the
  surface tension at zero tilt, or step free energy, and $H_p(L)$ is the entropic repulsion height, that is the typical height of the field when a positivity constraint is imposed. This generalizes recent results obtained in \cite{CMT}  for the standard SOS model ($p=1$). 
\end{abstract}


\maketitle

\section{Introduction}
Consider the $(2+1)$-dimensional generalized SOS model in a finite region  $\L\subset \bbZ^2$ with zero boundary conditions. This is the Gibbs probability measure $\bbP_\L$, on the set $\O^0_\L$ of 
discrete height functions $\phi:\bbZ^2\mapsto\bbZ$ such that $\phi(x)=0$ for all $x\notin \L$, defined by
\begin{equation} \label{eq:model}
\bbP_\Lambda(\phi) = \frac{1}{Z^0_{\Lambda}} \exp \left[ -\beta\sum_{|x-y|=1}|\phi(x)-\phi(y)|^p
\right], 
\end{equation}
$\phi\in\O^0_\L$, where  $\beta> 0$ is the inverse-temperature, $p\in [1,\infty]$ is a fixed parameter, and $Z^0_{\Lambda}$ is the normalizing partition function. We shall often consider the case where $\L$ is the lattice square  $\L_L:=[-L,L]^2\cap\bbZ^2$  with side $2L+1$ centered at the origin. 

Three values of $p$ correspond to well known models in statistical mechanics and
probability theory: the case $p=1$ is referred to as the standard \emph{SOS model} (cf.\ \cite{BMF,CRAS,CLMST2} and references therein), 
$p=2$ is the \emph{Discrete Gaussian model}
(cf.\ \cite{BMF,LMS} and references therein), while $p=\infty$ corresponds to the \emph{Restricted SOS model}, where only the values $\{-1,0,+1\}$ are allowed for the gradients $\phi(y)-\phi(x)$ along an edge  $(x,y)$
(cf. \cite{BMF} and also \cite{Peled}). 

A common feature of all  models defined by \eqref{eq:model} is that, for $\beta$ sufficiently large, the measures $\bbP_{\L_L}$ have a well defined limit $\bbP$ as $L\to\infty$, namely the  infinite volume Gibbs state with zero boundary
conditions, describing a strongly localized  interface with exponentially decaying correlations; see e.g. \cite{BW}.  Moreover, one can show that these models exhibit the so-called \emph{entropic repulsion} phenomenon. Namely, if one imposes the floor constraint
 \begin{equation}\label{posev}
 \{\phi(x) \geq
0,\;\text{ for every\;} x\in \L_L\},
\end{equation}
then, for $\b$ sufficiently large, with large probability  the field is pushed away from the floor and reaches a height $H_p(L)$, which diverges as $L\to\infty$. The phenomenon was first detected in \cite{BMF}, but precise results about the entropic repulsion height  were obtained only recently in
 \cite{CLMST,CRAS,CLMST2} for the SOS model ($p=1$), and in \cite{LMS} for the generalized SOS 
models ($p>1$). In particular, if one defines $H_p(L)$ by
\begin{equation}
 \label{eq-H-def}
  H_p(L) := \max\left\{ h\in \bbN : \bbP(\phi(0) \geq h) \geq 5\beta/L\right\}\,,
\end{equation}
then it is known that at low temperature for any $\d>0$ there exists a constant $K\in\bbN$ such that, given the event \eqref{posev}, the fraction of sites with height in $[H_p(L)-K,H_p(L)+K]$ is at least $1-\d$ with probability converging to $1$ as $L\to\infty$; see  \cite{CLMST,CRAS,CLMST2,LMS} for even sharper statements of this sort.      
Table \ref{tab1} below summarizes the large $L$ asymptotic behavior of $H_p(L)$ as $p$ varies; see also \cite{LMS}.
\begin{table}[h!]\label{tab1}
  \centering
\caption{Asymptotic behavior of $H_p(L)$, as $L\to \infty$.
  We write $\asymp a$ if $H_p(L)/a$ and $a/H_p(L)$ are $O(1)$; for the other entries $H_p(L)/a=1+o(1)$. 
}

{\setlength{\extrarowheight}{10pt}%
  \label{tab:table1}
  \begin{tabular}{|c|c|c|c|c|c|}
    \hline
$p=1$ & $1<p<2$ &$p=2$ &$2<p<\infty$ &
 $p=+\infty$ 
\\ [1ex] \hline 
$ \frac{1}{4\b}\log L$ & $\left(\frac{c(p)}{\b}\log
                                        L\right)^{1/p}$ & 
$ \sqrt{\frac{1}{4\pi \b}\log L\log\log L}$&
$\asymp \sqrt{\frac{1}{\b}\log L}$&
$\sqrt{\frac{1\pm \e_\b}{4\b +2\log\frac{27}{16}}\log L}$\\ [2ex]
\hline
  \end{tabular}}
\end{table}

In this paper we focus on the large
deviation asymptotic for the positivity event \eqref{posev}, as
$L\to \infty$. This question has been extensively studied for continuous height models such as the lattice massless free fields, see \cite{LebMaes,BDZ,BDG,DG,DGI}, and is a key problem in the study of the entropic repulsion phenomenon; see e.g.\ \cite{Yvan_survey} for a survey.
Our main result here reads as follows.

\begin{maintheorem}
\label{th:main}
There exists $\b_0>0$ such that for any $\b\geq \b_0$ and any $p\in
[1,+\infty]$ one has
\begin{align}\label{LD}
\lim_{L\to\infty}\frac1{8LH_p(L)}\log \bbP_{\L_L} \!\left(\phi(x) \geq
0\;\text{for every\;} x\in \L_L
\right) = -\b\, \t_{p,\b},
\end{align}
where the constant $\t_{p,\b}\in(0,\infty)$ is the surface tension at zero tilt (see Section \ref{prelim} below). Moreover, the same limit holds if we replace $\bbP_{\L_L}$ by the infinite volume Gibbs measure $\bbP$.
\end{maintheorem}

In the case of the standard SOS model, that is $p=1$, Theorem \ref{th:main} was proved in \cite[Theorem 1.2]{CMT}. As in that case, one can give the following heuristic explanation of the result: namely, the most convenient way to realize the event $ \phi(x) \geq
0$, for all $x\in \L_L$ is to lift the surface inside $\L_L$ up to the typical entropic repulsion profile, which entails the presence of   $H_p(L)(1+o(1))$ nested level lines each of length equal to $|\partial\L_L|(1+o(1))$, as $L\to\infty$, where $\partial\L_L$ denotes the boundary of the box $\L_L$; according to the large deviations theory each of these level lines has a cost roughly given by $e^{-\b\t_{p,\b}|\partial\L_L|}$; since $|\partial\L_L|=8L$, if these lines can be approximated by independent loops then  one obtains \eqref{LD}.  
To make this heuristic rigorous, the proof of \cite[Theorem 1.2]{CMT} used in a crucial way two facts: a) conditioned on the event \eqref{posev} the above described  scenario involving $H_p(L)(1+o(1))$ nested level lines is indeed very likely, and b) that the interaction between these level lines (beyond the natural order constraint) can be effectively neglected. 
Part a)  was essentially established in \cite{CLMST,CLMST2} for the SOS model. On the other hand part b) used a new recursive monotonicity argument based on FKG-type inequalities satisfied by the SOS models; see  \cite[Theorem 4.1]{CMT}.

The proof of Theorem \ref{th:main} will be based on a similar strategy. For part a) we shall rely on results obtained in \cite{LMS}, see Theorem \ref{mainthm:floor-shape} below. For part b) we shall need an extension to all $p\in[1,\infty]$ of the above mentioned monotonicity statements, see Theorem \ref{th:mon} below.     
From a technical point of view we remark that 
a key tool in our analysis is the so-called \emph{cluster expansion}. The standard version presented e.g.\ in \cite{BW} was adapted to the study of the statistics of the SOS level lines (contours) in \cite{CLMST,CLMST2}. Here we provide a nontrivial extension of those arguments to the generalized case $p>1$; see Section \ref{prelim} for the details.     
 
In order to keep the size of this note to a minimum, below we only give a detailed discussion of the points where the general case $p>1$ differs substantially from the case $p=1$. 
The remainder of the paper is as follows. Section \ref{prelim} gathers several basic preliminary facts concerning contours, cluster expansion and surface tension. Section \ref{sec:lower} proves the lower bound in Theorem \ref{th:main}. In Section \ref{sec:mon} we establish the needed monotonicity results. The upper bound in Theorem \ref{th:main} is discussed in Section \ref{sec:upper}.

\section{Contours, cluster expansion and surface tension}\label{prelim}
For any finite $\L\subset\bbZ^2$, let $\partial \L$ be the external boundary of
$\L$, i.e.\ the set of $y\in\L^c$ such that $xy$ is an edge of $\bbZ^2$, also called bond below, for some $x\in\L$. Let $\cB_\L\subset {\bbZ^2}$ denote
the set of $\bbZ^2$ edges  of the form $e=xy$ with $x\in\L$ and $y \in
\L\cup \partial \L$. A height configuration $\t: \L^c\mapsto \bbZ$ is called a boundary condition. 
We define $\O_\L^\t$ as the set of height functions $\eta:\bbZ^2\mapsto \bbZ$ such that $\eta(x)=\t(x)$ for all $x\notin \L$.   
For a fixed $p\in[1,\infty]$, the generalized SOS Hamiltonian in $\L$ with boundary condition $\t$ is the function defined by 
\begin{align}\label{psos_ham}
\cH_\L^\t (\phi) = \sum_{xy\in\cB_\L} |\phi(x)-\phi(y)|^p \,,\quad \;\phi\in\O_\L^\t.
\end{align}
The Gibbs measure in $\L$ with  boundary condition $\t$ at inverse temperature $\b$ is the probability measure 
$\bbP_\L^\t$ on $\O_\L^\t$ given by
\begin{align}\label{psos_gibbs}
\bbP_\L^\t(\phi) = \frac1{Z_\L^\t}\,\exp{\left(-\b\cH_\L^\t (\phi)\right)}\,,
\end{align}
where $Z_\L^\t$ is the partition function $$Z_\L^\t=\sum_{\phi\in\O_\L^\t}\exp{\left(-\b\cH_\L^\t (\phi)\right)}.$$ When $\t=0$ we simply write $Z_\L$ for $Z_\L^0$ and $\bbP_\L$ for $\bbP_\L^0$. 
Below, we consider $\L$ of rectangular shape, and write $\L_{L,M}$, with $L,M\in \bbN$, for the rectangle 
$\L_{L,M}=([-L,L]\times [-M,M])\cap\bbZ^2$ centered at the origin. When $L=M$ we  write $\L_L$ for the square of side $2L+1$. 

It is understood that the model at $p=\infty$ corresponds to case where the gradient term $|\phi(x)-\phi(y)|^p$ in \eqref{psos_ham} is replaced by $|\phi(x)-\phi(y)|\ind_{\phi(x)-\phi(y)\in\{-1,0,+1\}}$.
In the sequel, we will be working explicitly in the setting $p\in[1,\infty)$. The case $p=\infty$ can be recovered by taking the limit $p\to\infty$ in all steps of the proof.

%

\subsection{Geometric contours, and $h$-contours}\label{hcont}
We denote by ${\bbZ^2}^*\equiv \bbZ^2+(\frac12,\frac12)$  the dual graph of $\bbZ^2$. A dual edge or bond $e$ is seen as a segment connecting two neighboring vertices in ${\bbZ^2}^*$. Given $\L\subset
\bbZ^2$ we write $\phi_\L$ for the restriction of $\phi$ to $\L$, that is $\phi_\L=(\phi(x),\,x\in\L)$. We use the following notion of contours.
\begin{definition}\label{contourdef}
Two sites $x,y$ in $\bbZ^2$ are said to be {\em separated by a dual bond $e$} if
their distance (in $\bbR^2$) from $e$ is $\tfrac12$. A pair of
orthogonal dual
bonds which meet in a site $x^*\in {\Z^2}^*$ is said to form a
{\em linked pair of bonds} if both are on the same side of the
forty-five degrees line (w.r.t. to the horizontal axis) passing through $x^*$. A {\em geometric contour} (for
short, a contour in the sequel) is a
sequence $e_0,\ldots,e_n$ of dual bonds such that:
\begin{enumerate}
\item $e_i\ne e_j$ for $i\ne j$, except for $i=0$ and $j=n$.
\item for every $i$, $e_i$ and $e_{i+1}$ have a common vertex in ${\Z^2}^*$.
\item if $e_i,e_{i+1},e_j,e_{j+1}$ all have a common vertex $x^*\in {\Z^2}^*$,
then $e_i,e_{i+1}$ and $e_j,e_{j+1}$ are linked pairs of bonds.
\end{enumerate}
We denote by $|\gamma|$ the length of a contour $\gamma$, that is the number of distinct bonds in $\g$.
If $e_0=e_n$ we say that the contour is \emph{closed}, otherwise it is
\emph{open}.  If $\g$ is closed, then its interior
(the sites in $\bbZ^2$ it surrounds) is denoted by $\Lambda_\gamma$ and its
interior area (the number of such sites) by
$|\Lambda_\gamma|$. For a closed $\g$, we let $\Delta_{\gamma}$ be the set of sites in $\Z^2$ such that either their distance
(in $\bbR^2$) from $\gamma$ is $\tfrac12$, or their distance from the set
of vertices in ${\Z^2}^*$ where two non-linked bonds of $\gamma$ meet
equals $\tfrac1{\sqrt2}$. Finally, we set $\Delta^+_\gamma=\Delta_\gamma\cap
\L_\g$ and $\Delta^-_\gamma = \Delta_\gamma\setminus \Delta^+_\gamma$. We observe that for any $x\in \D^+_\g$ there exists $y\in \D_\g^-$ such that either $x,y$ are nearest neighbor or their distance is $\sqrt{2}$.  
Given a closed contour $\gamma$ we say that $\gamma$ is an \emph{$h$-contour}
for the configuration $\phi$ if
\[
\phi_{\Delta^-_\gamma}\leq h-1, \quad \phi_{\Delta^+_\gamma}\geq h.
\]
Finally $\sC_{\gamma,h}$ will denote the event that $\gamma$ is an $h$-contour.
\end{definition}

Let us illustrate the above definitions with some simple examples. Consider the elementary contour $\g$ given by the unit square surrounding a site $x\in\bbZ^2$:  $\g$ is an $h$-contour iff $\phi({x})\geq h$ and $\phi(y)\leq h-1$ for all $y\in\{x\pm e_1, x\pm e_2, x+ e_1+e_2,x-e_1-e_2\}$. 
Another explicit example is given in Figure \ref{fig:002} below. We observe that 
a geometric contour $\g$ could be at the same time a
$h$-contour and a $h'$-contour with $h\neq h'$ for some configuration $\phi$. More generally two
geometric contours $\g,\g'$ could be contours for the same surface
with different height parameters even if $\g\cap\g'\neq\emptyset$, but then
the interior of one of them must be contained in the interior of the other. 

\begin{figure}[htb]
        \centering
 \begin{overpic}[scale=0.37]{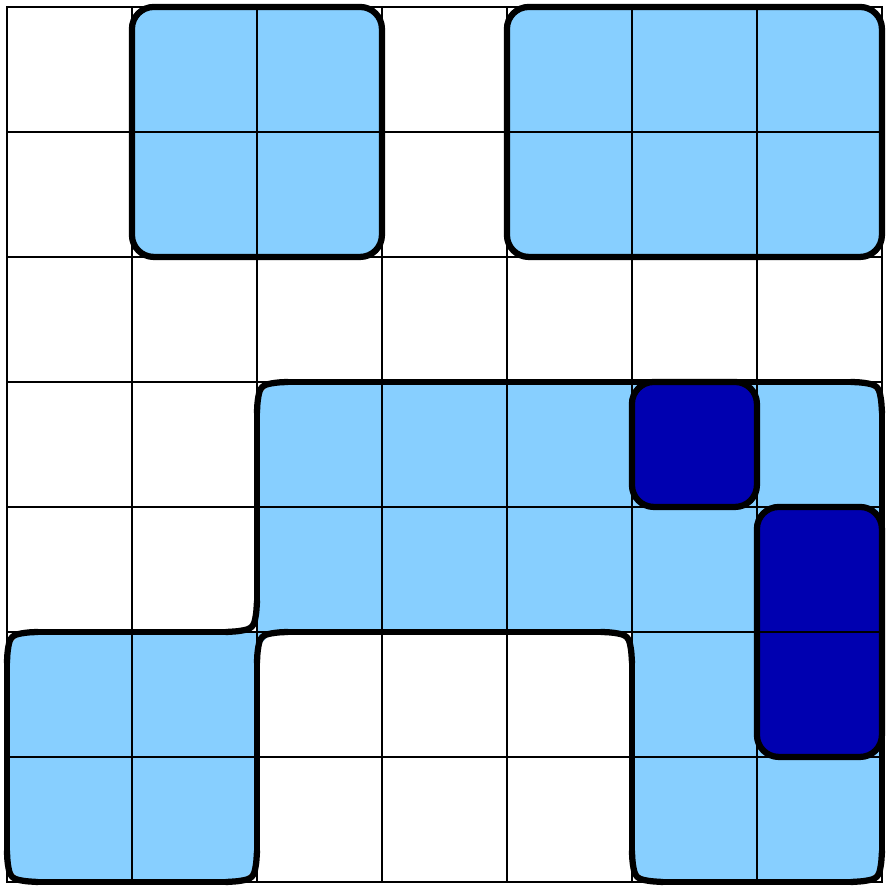}
\end{overpic}
        \caption{
        Example of a surface configuration $\phi$  in the $7\times 7$ box $\L_3$ with zero boundary conditions: white sites have height $0$, shaded sites have height $1$ and darker sites have height $2$. According to Definition \ref{contourdef} there are three $1$-contours and two  $2$-contours in this configuration.
          }\label{fig:002}
\end{figure}

\subsection{Cluster expansion in presence of sign constraints}\label{clexp}
Let $\g=(e_0,e_1,\dots,e_n)$ be an open
contour such that the first and last bonds of
$\g$ are horizontal and cut the two opposite vertical sides of a rectangle $\L\subset\bbZ^2$,  while all the other bonds are contained in $\L$. For such a $\g$ we define the sets $\D^\pm$ as follows, see also Figure \ref{fig:013}: connect the two endpoints of $\g$ by adding dual bonds that stay outside the rectangle $\L$, so that one has now a closed contour $\g'$, and set $\D^\pm := \D^\pm_{\g'}\cap \L$, where $\D^\pm_{\g'}$ is defined in Definition \ref{contourdef}.
For $xy\in\cB_\L$, $a\geq 0$, we define the function $h_{xy}(a)=(1+a)^p-1$ if $xy$ crosses a dual bond $e_i$ of $\g$, and  
$h_{xy}(a)=a^p$ otherwise. 
Consider the constrained partition function 
\begin{equation}\label{zod}
Z^0_{\L,\D^+,\D^-}=\sum_{\phi\in\O_\L^0(\D^\pm)}\exp\left[-\beta
\sum_{xy\in\cB_\L}h_{xy}(|\phi(x)-\phi(y)|)\right], \quad
\end{equation}
where $\O_\L^0(\D^\pm)$ denotes the set of all $\phi\in\O^0_\L$ that satisfy the constraints $\D^\pm$, that is $\phi_{\D^+}\ge 0$, and 
$\phi_{\D^-}\le 0$. We call such surfaces {\em legal}.
\begin{figure}
        \centering
 \begin{overpic}[scale=0.45]{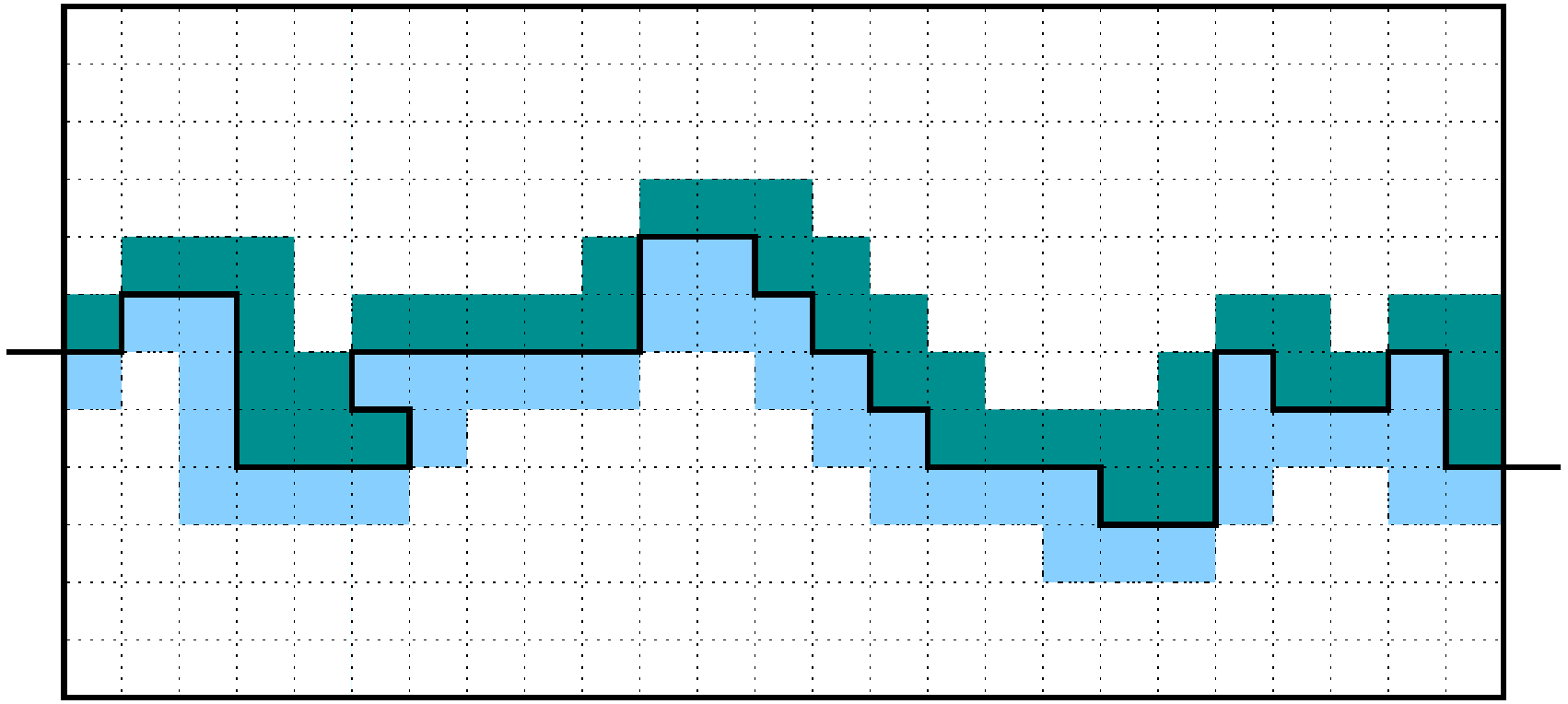}
\put(46,11){$\D^-$}
\put(57,29){$\D^+$}
\end{overpic}
        \caption{The rectangle $\L$ with an open contour $\g$. The sets $\D^+$ and $\D^-$ are respectively the shaded regions above and below the contour $\g$.}\label{fig:013}
\end{figure}

The partition function \eqref{zod} arises naturally in the following important setting. 
Consider the rectangle $\L=\L_{L,M}$, for some $L,M\in \bbN$. 
Fix integers $a,b\in[-M,M]$, 
and define the boundary condition $\t$ on $\partial \L_{L,M}$ so that 
$$
\t(u,v)=1_{\{u=-L-1,\,v\geq a\}} + 1_{\{u=L+1,\,v\geq b\}} + 1_{\{-L\leq u\leq L,\, v=M+1\}},
$$
for all $(u,v)\in\partial \L_{L,M}$.
Let $Z(a; b; L,M):=Z_{\L}^\t$ denote the partition function of the $SOS$ model in $\L$ with boundary condition $\t$, and note that for every $\phi$ contributing to $Z_{\L}^\t$ there must exist an open contour $\g$ crossing $\L$ from left to right as above, joining the dual points $(-L-1/2,a-1/2)$ and $(L+1/2,b-1/2)$, such that $\g$ is a $1$-contour for $\phi$. 
Then, adding and subtracting $1$ for each dual edge in $\g$ and changing variable to $\phi-1$ in the region above $\g$ one finds the expression
\begin{equation}\label{zetod}
Z(a; b; L,M) = \sum_{\g} e^{-\b|\g|}Z^0_{\L,\D^+,\D^-},
\end{equation}
where the sum is over all possible open contours as above and $Z^0_{\L,\D^+,\D^-}$ is given in \eqref{zod}. 

In 
\cite[Lemma 2.2]{CMT} and \cite[Appendix A]{CLMST} we have obtained a cluster expansion for \eqref{zod} in the case where $p=1$ (so that $h_{xy}(a)=a$ for all ${xy}$) and  the two sets $\D^\pm$ are
two arbitrary subsets of $\partial_*\L$, where  $\partial_*\L$ is defined as the set of points of $\L$ that are either at distance $1$ from $\partial \L$, or at distance $\sqrt{2}$ from $\partial \L$ in the South-West or North-East direction. Here we need the following extension of that.  
\begin{proposition}
\label{lem:dks}
There exists $\b_0>0$ independent of $\L$ and $p\in[1,\infty]$, such that if $\b\geq \b_0$, then \eqref{zod} satisfies
\begin{equation}\label{zetaexp}
\log  Z^0_{\L,\D_+,\D_-}=\sum_{V\subset \L}\varphi_{\D^+,\D^-}(V),
\end{equation}
where the potentials $\varphi_{\D^+,\D^-}(V)$ are such that
\begin{enumerate}
\item $\varphi_{\D^+,\D^-}(V)=0$ if $V$ is not connected.
\item $\varphi_{\D^+,\D^-}(V)=\varphi_0(V)$ if $V\cap({\D^+\cup \D^-})=\emptyset$, for
  some shift invariant  potential $V\mapsto \varphi_0(V)$,
that is
\[
\varphi_0(V)=
\varphi_0(V+x)\quad \forall \,x\in\bbZ^2,
\]
where $\varphi_0$ is independent of $\D^\pm,\L$.
\item For all $V\subset \L$:
\[
\sup_{\D^\pm}|\varphi_{\D^+,\D^-}(V)|\leq
  \exp(-(\b-\b_0)\, d(V))
\]
where $d(V)$ is the cardinality of the smallest connected set of all dual
bonds separating points of $V$ from points of its
complement (a dual bond is said to separate $V$ from $V^c$ iff it is orthogonal to a bond connecting $V$ to $V^c$). 
\end{enumerate}
\end{proposition}
The core of the proof is to show that the ``usual" cluster representation
(e.g.\ \cite[Corollary 2.2]{BW}) continues to hold for the constrained partition
function  \eqref{zod}. This nontrivial fact follows from Lemma
\ref{lem:ce2} below. Once this lemma is
established the rest of the proof is obtained as in \cite{CMT} and
\cite[Appendix A]{CLMST}. 
Following \cite{BW} we first
define what we mean by clusters. We assume that an ordering of all dual bonds has been defined, and that an orientation for each bond of $\bbZ^2$ has been fixed. Both choices are arbitrary. 
\begin{definition}\ 
\begin{enumerate}
\item A \emph{cluster} $X=\left(\G,\grad_1,\grad_2,\dots,\grad_n\right)$
  consists of a finite connected set of $n$ dual bonds $\G$, called
  the \emph{support of $X$}, and certain height gradients $\grad_i\in
  \bbZ$ for each bond orthogonal to a dual bond in $\G$, appearing in the prescribed dual bond ordering. The set $\G$ is such that
  any endpoint of one of its bonds is shared by at least another bond in $\G$.
  The support of $X$ divides $\bbR^2$ into an infinite component and a collection of finite connected sets called the \emph{faces of $X$}. The interior ${\rm Int}(X)\subset \bbR^2$ is the union of the finite faces of $X$. The external boundary $\partial_{\rm ext}X$ consists of those points of $\bbZ^2\cap(\bbR^2\setminus {\rm Int}(X))$ which have a nearest neighbor in ${\rm Int}(X)$. 
\item A \emph{configuration} of clusters is a unordered finite collection of clusters $\{X_1,\dots,X_n\}$. A configuration is \emph{compatible}  if the corresponding supports $\G_i$ are mutually disjoint.
\item A cluster $X$ is \emph{admissible} if there exists a unique surface $\Phi=\Phi(X):\bbZ^2\mapsto\bbZ$ such that $\Phi_{\partial_{\rm ext}(X)}=0$ and the gradients of $\Phi$ coincide with those specified by $X$ along each bond orthogonal to a dual bond in $\G$ and are zero on every bond that does not intersect $\G$. By construction $\Phi$ is constant on each face of $X$. 

\item In the setting of Proposition \ref{lem:dks}, a cluster $X$ is called  \emph{legal} if it is admissible and the associated surface $\Phi$ satisfies the  constraints on $\D^\pm$, that is $\Phi\in\O_\L^0(\D^{\pm})$.   
\end{enumerate}
\end{definition}
In \cite{BW} it was proved that the set of surfaces with zero
boundary conditions, but without the constraints on $\D^\pm$, is in
one-to-one correspondence with the set of all  compatible configurations of admissible
clusters. In Lemma \ref{lem:ce2} below we extend the result to the constrained case.
We start with a technical lemma. 
\begin{lemma}
\label{lem:ce1}
In the setting of Proposition \ref{lem:dks}, given a legal cluster $X$ with support $\G$ let $R$ be one of the
faces of $X$. If
$R\cap \D^+\neq \emptyset$ and $R\cap \D^-=\emptyset$ then necessarily
$R\cap \D^+\subset \partial_{*}R$.
The same applies by inverting the role of $\D^+$ and $\D^-$. If instead $R\cap \D^+\neq \emptyset$ and $R\cap \D^-\neq \emptyset$, then $\Phi(X)_R=0$.      
\end{lemma}
\begin{proof}[Proof of Lemma \ref{lem:ce1}]
Fix a legal cluster $X$ together with one of its faces $R$. By construction, the surface height $\Phi(X)$ is constant on $R$. Suppose first that there exist $x,y\in R$ such that $x\in \D^+$ and $y\in \D^-$. Then $\Phi(X)_R=0$ since it is non-negative at $x$ and non-positive at $y$. Assume now that $R\cap \D^+\neq \emptyset,\ R\cap \D^-=\emptyset$ and let $x\in R\cap \D^+$. By construction, (cf.\ Definition \ref{contourdef}) there exists a point $y\in \D^-$ which is either a nearest neighbor of $x$ or is at distance $\sqrt{2}$ from $x$ in the South-West direction. By assumption $y\notin R$ and therefore $x\in \partial_* R$.   The same argument applies with $\D^+,\D^-$ interchanged.
\end{proof}

\begin{lemma}
\label{lem:ce2}
The set of legal surfaces in \eqref{zod} is in one-to-one correspondence with the set of all compatible configurations of  legal clusters $\{X_1,X_2,\dots, X_n\}$.    In particular, \eqref{zod} can be written as
\begin{equation}\label{zod1}
Z^0_{\L,\D^+,\D^-}=\sum_{\{X_1,\dots,,X_n\}}\,\prod_{i=1}^n
\exp\left[-\b \sum_{b\in\G_i}h_{b^\perp}\left(|\grad^{(i)}_{b^\perp}|\right)\right],
\end{equation}
where $\G_i$ denotes the support of
$X_i$, $\grad^{(i)}_{b^\perp}$ is the height gradient in $X_i$ associated to $b^\perp $, the lattice bond orthogonal to bond $b\in\G_i$, and the first sum runs over all compatible configurations of legal clusters $\{X_1,X_2,\dots, X_n\}$. 
\end{lemma}
\begin{proof}[Proof of Lemma \ref{lem:ce2}]
We begin with the proof that any  compatible configuration of legal clusters defines a legal surface. We proceed by induction. If $n=1$, it is true by definition. Suppose the statement true up to $n-1$ and fix a  compatible configuration of  legal clusters $\{X_1,X_2,\dots, X_n\}$. Without loss of generality, possibly by a suitable relabeling of the clusters, we can assume that either ${\rm Int}(X_n)\not\subset \cup_{j=1}^{n-1}{\rm Int}(X_j)$ or ${\rm Int}(X_n)\subset {\rm Int}(X_{n-1})$ and ${\rm Int}(X_n)\cap\, \cup_{j=1}^{n-1}\G_j=\emptyset$. Call $\Phi$ the surface defined by $\{X_1,X_2,\dots, X_{n-1}\}$ and $\Phi'$ that obtained after the insertion of $X_n$ and whose existence is guaranteed by the results in \cite{BW}. In the first case there is nothing to prove, since by the inductive assumption $\Phi$ is legal and the addition of $X_n$ outside of $\cup_{j=1}^{n-1}{\rm Int}(X_j)$ creates a legal $\Phi'$. In the second case let $R$ be the face of $X_{n-1}$ containing $X_n$. 
\begin{itemize}
\item If $R\cap \D^\pm=\emptyset$ then the addition of $X_n$ to $\Phi$ does not change the values of $\Phi$ on $\D^+\cup\D^-$ and $\Phi'$ is also legal.
\item If $R\cap \D^+\neq \emptyset$ and $R\cap \D^- =\emptyset$ then, by Lemma \ref{lem:ce1}, $\D^+\subset \partial_*R$ and therefore $\D^\pm\cap {\rm Int}(X_n)=\emptyset$. Hence we conclude as in case (i). The same applies by inverting the role of $\D^+$ and $\D^-$.
\item If $R\cap \D^+\neq \emptyset$ and $R\cap \D^-\neq \emptyset$ then, by Lemma \ref{lem:ce1},  $\Phi$ must be equal to zero in $R$ and therefore we can safely add $X_n$ to $\Phi$.
\end{itemize}
We now prove the converse. Using \cite{BW}
we know that any legal surface corresponds to a unique  compatible configuration of admissible clusters $\{X_1,X_2,\dots, X_n\}$. It remains to check that each $X_i$ is legal. 

By induction we can assume the statement true for any legal surface
corresponding to $(n-1)$ admissible and compatible clusters. As before
we can assume that either (a) ${\rm Int}(X_n)\not\subset
\cup_{j=1}^{n-1}{\rm Int}(X_j)$ or (b) ${\rm Int}(X_n)\subset {\rm Int}(X_{n-1})$ and ${\rm Int}(X_n)\cap\, \cup_{j=1}^{n-1}\G_j=\emptyset$. 
Let $\Phi$ be the surface defined by the first $n-1$ clusters. We have to distinguish between two cases. 
\begin{enumerate}
\item $\Phi$ is legal. In this case by induction each $X_j$, $j\le
  n-1$, is legal. In case (a) necessarily $X_n$ must be legal. In
  case (b) let $R$
  be the face of $X_{n-1}$ containing ${\rm Int}(X_n)$. If $R\cap \D^\pm=\emptyset$ then $X_n$ is legal since it is admissible. If $R\cap \D^+\neq \emptyset$ and $R\cap \D^- =\emptyset$ then $\D^+\subset \partial_*R$ and therefore $\D^\pm\cap {\rm Int}(X_n)=\emptyset$. Hence we conclude as before. The same applies by inverting the role of $\D^+$ and $\D^-$.
If $R\cap \D^+\neq \emptyset$ and $R\cap \D^-\neq \emptyset$ then, using Lemma \ref{lem:ce1}, $\Phi$ must be equal to zero in $R$. Therefore, since the addition of $X_n$ to the flat surface at height zero produces a legal surface, $X_n$ must be legal.
\item $\Phi$ is not legal, but adding an admissible $X_n$ produces a legal
surface. We shall prove that this scenario cannot happen. In case (a)  adding an admissible $X_n$ cannot produce a legal
surface. In case (b) let $R$ be the face of $X_{n-1}$ containing $X_n$. We can safely assume that $\Phi$ violates the constraints ($\D^\pm$) inside $R$, since otherwise $X_n$ will not be able to restore them.
If $R\cap \D^+\neq \emptyset$ and $R\cap \D^-=\emptyset$ or viceversa then, as before, by adding $X_n$ we cannot change the value of the surface on $\D^+\cup \D^-$ and we get a non legal surface. Suppose now that $R$ contains points of $\D^+$ and of $\D^-$ and assume that e.g. $\Phi=h>0$ inside $R$. By the nature of the open contour producing $\D^\pm$, even if we add the cluster $X_n$ we cannot lower the value of $\Phi$ everywhere inside $R$ because $X_n$ is compatible with $X_{n-1}$. Thus the surface obtained by adding $X_n$ cannot be legal. 
\end{enumerate}
This establishes the claimed one-to-one correspondence. The expression \eqref{zod1} then follows by definition of compatible legal clusters. This ends the proof of Lemma \ref{lem:ce2}.
\end{proof}

\subsection{Surface tension}\label{surface_tension}
The expansion of Proposition \ref{lem:dks} is a version of the well known cluster expansion for 2D Ising  contours; see e.g.\ \cite{DKS}. In particular, it allows one to deduce a number of facts about 
the model, including the existence and basic properties of the surface tension, a.k.a.\ step free energy in this setting. Consider the partition function $Z(a; b; L,M)$ defined in \eqref{zetod}, and recall that $Z_\L$ denotes the partition function with zero boundary conditions and no sign constrains. One can show that the limit 
\begin{align}\label{zst}
\cZ(a; b; L)=\lim_{M\to\infty}\frac{Z(a; b; L,M)}{Z_\L} 
\end{align}
exists for every $L$, and fixed $a,b\in\bbZ$, and satisfies
\begin{align}\label{contourmod}
 &\cZ(a;b;L) =
\sum_{\g}e^{-\b|\g| + 
 \Phi_{L,\infty}(\g)}\,,
\end{align}
where the sum ranges over all open contours in the strip $\L_{L,\infty}=[-L,L]\times(-\infty,\infty)\cap\bbZ^2$ 
joining the dual lattice points
$x:=(-L-1/2,a-1/2)$ and $ y:=(L+1/2,b-1/2)$, and $ \Phi_{L,\infty}(\g)$ is the usual ``decoration" term; see e.g.\ \cite{DKS}. Moreover, one can prove the following facts; we refer to \cite{CMT} and references therein for more details. 
 \begin{lemma}\label{lem:surftens}
There exists $\b_0>0$ such that the following holds for all $\b\geq \b_0$, $p\in[1,\infty]$. 
Assume that as $L\to\infty$ one has $(b-a)/(2L)\to \l\in\bbR$ and set $\theta=\tan^{-1}(\l)$.   Then, the limit
\begin{equation}
    \label{eq:surftens1}
\t_{p,\b}(\theta)=-\lim_{L\to \infty}\frac{\cos(\theta)}{2\b L}\log \cZ(a;b;   L),
  \end{equation}
 is well defined and positive in $(-\pi/2,\pi/2)$. 
 \end{lemma} 
Equation \eqref{eq:surftens1} defines the  \emph{surface tension} at angle $\theta$. When $\theta=0$ we simply write $ \t_{p,\b}=\t_{p,\b}(0)$. We refer to \cite[Lemma 2.4]{CMT} for further important properties of this function that continue to hold in our setting.  

\section{Lower bound}\label{sec:lower}
As in  \cite[Lemma 3.1]{CMT} it is not difficult to show that for $\b\geq \b_0$ the infinite volume measure $\bbP$ satisfies \begin{align}\label{pospo1}
\bbP({\phi_{\partial\L_{L}}}=0)\geq e^{-c\, L},
\end{align} for some constant $c>0$.
In particular, since $H_p(L)\to\infty$ as $L\to\infty$, it follows that if  we prove the lower bound for $\bbP_{\L_L}$ in Theorem \ref{th:main},  
we also have the same lower bound for $\bbP$ by using 
\begin{align}\label{posp1}
\bbP( \phi_{\L_L}\geq 0) \geq  \bbP({\phi_{\partial\L_{L}}}=0;\phi_{\L_L}\geq 0) =  \bbP({\phi_{\partial\L_{L}}}=0)\bbP_{\L_L} ( \phi_{\L_L}\geq 0).
\end{align} 
To prove the lower bound for $\bbP_{\L_L}$ we proceed by restricting the set of configurations to a suitable 
event. As in \cite{CMT}, to define this event we need to introduce the following geometric construction and the associated set of nested contours. 

Fix $ N:=H_p(L)$, and
define the nested annular regions 
$\bar \cU_i:=\L_{L-3\ell_{i-1}}\setminus\L_{L-3\ell_i}$, $i=1,\dots, N$, where $\ell_0=0$ and $\ell_i=i(i+1)/2$. The set $\bar \cU_i$ consists of 3 nested disjoint annuli each of width $i$. We call $\cU_i$  the middle one, i.e.\ $\cU_i =  \L_{L-(3\ell_{i-1} -i)}\setminus\L_{L-(3\ell_i+i)}$. Notice that $d(\cU_i,\cU_{i+1})\geq 2i +1$, where $d(\cdot,\cdot)$ is the euclidean distance. 

For each $i$, let $\cC_i$ denote the set of all closed contours $\g$ 
such that $\g\subset\cU_i$ and $\g$ surrounds $\L_{L-(3\ell_i+i)}$. 
For fixed $\g\in\cC_i$, let $\sC_{\g,i}$ denote the event that the surface $\phi$ has an $i$-contour at $\g$. 
Define the event $E$ that for all $i=1,\dots,N$ there exists an $i$-contour $\g_i\in\cC_i$, that is
\begin{align}\label{eveE}
E := \cup_{\g_1\in\cC_1,\dots,\g_N\in\cC_N}\,\sC_{\g_1,1}\cap\cdots\cap\sC_{\g_N,N}.
\end{align} 
Fix a choice of  contours $\g_i\in \cC_i$, $i=1,\dots, N$, and consider the probability measure $\wt\bbP_{\L_L}=\wt\bbP_{\L_L,\g_1,\ldots,\g_N}$ 
defined as follows. Define the modified hamiltonian
\begin{align}\label{tsos_ham}
\wt\cH_{\L_L,\g_1,\ldots,\g_N} (\phi) = \sum_{xy\in\cB_\L} h_{xy}(|\phi(x)-\phi(y)|) \,,
\end{align}
where 
$h_{xy}(a)=(1+a)^p-1$ if $xy$ crosses a dual bond $e\in\cup_{i=1}^N\g_i$, and  
$h_{xy}(a)=a^p$ otherwise.  Then define  
\begin{align}\label{tsos_part}
\wt Z_{\L_L,\g_1,\ldots,\g_N}=\sum_{\phi\in\O^0_{\L_L,\g_1,\ldots,\g_N}}e^{-\b\wt\cH_{\L_L,\g_1,\ldots,\g_N} (\phi)}  \,,
\end{align}
where $\O^0_{\L_L,\g_1,\ldots,\g_N}$ denotes the set of all $\phi\in\O^0_{\L_L}$ such that $\phi_{U^+}\geq 0$ and $\phi_{U^-}\leq 0$, where $U^\pm:=\cup_{i=1}^N\D_{\g_i}^\pm$. The probability $\wt\bbP_{\L_L}$ on $\O^0_{\L_L,\g_1,\ldots,\g_N}$ is defined by
 \begin{align}\label{tsos_prob}
\wt\bbP_{\L_L}(\phi) = \frac{e^{-\b\wt\cH_{\L_L,\g_1,\ldots,\g_N} (\phi)}}{\wt Z_{\L_L,\g_1,\ldots,\g_N}}.\end{align}
 Define also $S_i:=\L_{\g_{i-1}}\setminus \L_{\g_i}$, with the convention that $\L_{\g_0}=\L_L$ and $\L_{\g_{N+1}}=\emptyset$.
With this notation, the change of variable $\phi(x)\mapsto\phi(x)-i+1$, for each $x\in S_i$, yields
 \begin{align}\label{elbo4}
&\bbP_{\L_L}( \phi_{\L_L}\geq 0\,;\,\sC_{\g_{1},1}\cap\cdots\cap\sC_{\g_N,N})
\nonumber\\&\qquad =e^{-\b\sum_{i=1}^N|\g_i|}\,\frac{\wt Z_{\L_L,\g_1,\ldots,\g_N}}{Z_{\L_L}}\,\wt\bbP_{\L_L}(\cap_{i=1}^{N+1}\{\phi_{S_i}\geq -i+1\}),
\end{align}
where $\phi_{A}\geq a$ is shorthand notation for: $\phi(x)\geq a$ for all $x\in A$. 

We start by estimating the probability appearing in the right hand side of \eqref{elbo4}. 
\begin{lemma}\label{lemelbo} 
There exists a constant $C>0$ such that for any choice of contours $\g_i\in\cC_i$, $i=1,\dots,N$, one has  
\begin{align}\label{lb51}
\wt\bbP_{\L_L}(\cap_{i=1}^{N+1}\{\phi_{S_i}\geq -i+1\})\geq  e^{-C L}
\end{align}
\end{lemma}
\begin{proof}
One verifies that the FKG lattice condition is satisfied by the measure  $\wt\bbP_{\L_L}$, that is $\wt\bbP_{\L_L}(\phi\vee\phi')\wt\bbP_{\L_L}(\phi\wedge\phi')\geq \wt\bbP_{\L_L}(\phi)\wt\bbP_{\L_L}(\phi')$ for all $\phi,\phi'\in\O^0_{\L_L,\g_1,\ldots,\g_N}$. Indeed, 
it suffices to show that 
\begin{align}\label{convo}
h_{xy}&(|\max\{a,c\}-\max\{b,d\}|)+ h_{xy}(|\min\{a,c\}-\min\{b,d\}|) \\&\qquad\qquad \qquad\leq h_{xy}(|a-b|) + h_{xy}(|c-d|)\,,
\end{align}
for all $a,b,c,d\in\bbR$, and for all $xy\in\cB_\L$. To prove \eqref{convo}, define the function $\bbR\ni t\mapsto \o(t):=h_{xy}(|t|)$, and note that
since $p\geq 1$, $\o$ is convex. Let us consider e.g.\ the case $a\geq c\geq b$, $c\leq d$, and $d\geq b$. All other cases can be handled in a similar way. 
The left hand side of \eqref{convo} becomes $\o(a-d)+\o(c-b)$. Setting $\a=(a-c)/(d-c+a-b)$, by 
convexity $\o(a-d)\leq\a\o(c-d)+(1-\a)\o(a-b)$, and $\o(c-b)\leq(1-\a)\o(c-d)+\a\o(a-b)$. Therefore, summing up one obtains $\o(a-d)+\o(c-b)\leq \o(c-d)+\o(a-b)$, as in \eqref{convo}. 
 
 Then, the FKG inequality implies 
\begin{align*}
\wt\bbP_{\L_L}(\cap_{i=1}^{N+1}\{\phi_{S_i}\geq -i+1\})\geq
\prod_{i=1}^{N+1}\prod_{x\in S_i}\wt\bbP_{\L_L}(\phi(x)\geq -i+1)\end{align*}
The argument from \cite[Proposition 3.9]{CLMST} can be adapted to the case $p>1$ 
to obtain 
\begin{align*}
\wt\bbP_{\L_L}(\phi(x)\geq -i+1)\geq 1-C\,\bbP(\phi(0)\geq i), 
\end{align*}
for some constant $C>0$ and for all $i$. 
Using $1-t\geq e^{-2t}$ for $0\leq t\leq 1/2$, one has
   \begin{align}\label{lbo51}
\wt\bbP_{\L_L}(\cap_{i=1}^{N+1}\{\phi_{S_i}\geq -i+1\})\geq
 e^{-C-2\sum_{i=1}^{N+1} C |S_i| \bbP(\phi(0)\geq i)},
\end{align} for some new constant $C>0$.
 Estimating $|S_i|\leq Ci L$ one finds 
 $\sum_{i=1}^{N} |S_i|\bbP(\phi(0)\geq i)\leq CL$.
 On the other hand, the term $|S_{N+1}|\bbP(\phi(0)\geq N+1)= |\L_{\g_N}|\bbP(\phi(0)\geq H_p(L)+1)$ satisfies
 \begin{align}\label{lb533}
|S_{N+1}|\bbP(\phi(0)\geq N+1)\leq 4L^2\bbP(\phi(0)\geq H_p(L)+1)\leq C\,\b\, L ,
 \end{align}
 where we use the definition \eqref{eq-H-def} of $H_p(L)$. This proves \eqref{lb51}.
\end{proof}

Next, we proceed as in \cite[Lemma 3.2]{CMT}, that is we define $n:=\lfloor\log \log L\rfloor$, fix arbitrary contours $\g^*_1\in\cC_1,\dots,\g^*_n\in\cC_n$, and then sum over  the remaining contours $\g_i\in\cC_i$, $i=n+1,\dots,N$. 
 \begin{lemma}\label{lemlb}
 Fix $\b\geq \b_0$ and fix  $\g^*_1\in\cC_1,\dots,\g^*_n\in\cC_n$, where $n=\lfloor\log \log L\rfloor$. Then
  \begin{align}\label{lemlb1}
 \bbP_{\L_L}( \phi_{\L_L}\geq 0; E) \geq\; \tfrac12
 \sum_{\g_{n+1}\in\cC_{n+1},\dots,\g_N\in\cC_N}
 \bbP_{\L_L}( \phi_{\L_L}\geq 0\,;\,\cap_{k=1}^n\sC_{\g^*_{k},k}\,;\,\cap_{j=n+1}^N\sC_{\g_{j},j}
 ).
 \end{align}
 \end{lemma}
\begin{proof}
Since the measure $\wt\bbP_{\L_L}$ satisfies the FKG property (cf.\ Lemma \ref{lemelbo} above)
one can use monotonicity with respect to boundary conditions in the regions $S_i$ and the 
proof follows as in \cite[Lemma 3.2]{CMT}, with the only difference that when $p>1$ the reference to \cite[Proposition 2.7]{CLMST2} must be replaced with \cite[Proposition 4.1]{LMS}. 
\end{proof}
From Lemma \ref{lemlb}, we see that  the lower bound in Theorem \ref{th:main} follows if,  as $L\to\infty$:
\begin{align}\label{lb7}
\sum_{\g_{n+1}\in\cC_{n+1},\dots,\g_N\in\cC_N}
&  \bbP_{\L_L}( \phi_{\L_L}\geq 0\,;\,\cap_{k=1}^n\sC_{\g^*_{k},k}\,;\,\cap_{j=n+1}^N\sC_{\g_{j},j})  \nonumber \\
  & \qquad\qquad \quad \geq \exp{\big(-8\b\t_{p,\b}NL(1+o(1))\big)},
 \end{align}
 uniformly in the fixed choice of $\g^*_k\in\cC_k$, $k=1,\dots,n$, with $n=\lfloor\log \log L\rfloor$.

Thanks to Lemma \ref{lemelbo} we can neglect the term $\wt\bbP_{\L_L}(\cap_{i=1}^{N+1}\{\phi_{S_i}\geq -i+1\})$ when proving \eqref{lb7}.
The proof of the lower bound is then concluded by showing that uniformly in the choice of $\g^*_k\in\cC_k$, $k=1,\dots,n$, with $n=\lfloor\log \log L\rfloor$, one has 
\begin{align}\label{elbo41}
&\sum_{\g_{n+1}\in\cC_{n+1},\dots,\g_N\in\cC_N}e^{-\b\sum_{i=1}^N|\g_i|}\,\frac{\wt Z_{\L_L,\g^*_1,\dots,\g^*_n,\g_{n+1}\ldots,\g_N}}{Z_{\L_L}} \nonumber\\ &\qquad\qquad\qquad\qquad \geq
\exp{\big(-8\b\t_{p,\b}NL(1+o(1))\big)}.
\end{align}
The estimate \eqref{elbo41} can be obtained with the same argument as in \cite[Lemma 3.3, Lemma 3.4]{CMT}. The only technical issue is that the cluster expansion used there - see \cite[Eq. (3.8)]{CMT} has to be adapted to the case $p>1$. However, this adaptation can be obtained by minor modifications of the arguments used in Section \ref{clexp} and Section \ref{surface_tension} above. We omit the details.

\section{A monotonicity property}\label{sec:mon}
In this section we extend to the case $p>1$ the monotonicity property of the partition
function which played a crucial role in the proof of the large
deviation result for the SOS model; cf.\ \cite[Theorem
4.1]{CMT}. In order to formulate it we need to recall the definition of
the \emph{staircase ensemble} introduced in \cite{CMT}.

\subsection{Staircase ensemble}\label{stairs}
In \eqref{zetod} we considered the case of a stepped boundary condition, which produced one open contour. Here we generalize that to the case of a multi-step boundary condition, which is associated with the presence of $n$ open contours. Consider the rectangle $\L_{L,M}$, for some $L,M\in \bbN$. 
Fix $n\in\bbN$ and integers 
\begin{equation}
\label{ref2}
-M\leq a_1\leq \cdots\leq a_n\leq M,\;\;\text{ and}\;\; -M\leq b_1\leq \cdots\leq b_n\leq M,  
\end{equation}
and set $a_0=b_0=-(M+1)$ and $a_{n+1}=b_{n+1}=M+1$. We define a ``staircase'' height $\t$ at the external boundary $\partial \L_{L,M}$ of our rectangle which, starting from 
height zero at the base of the rectangle (i.e. the set $(u,-(M+1)), u=-L,\dots,L$)  jumps by one at the locations specified by the two
$n$-tuples $\{a_i,b_i\}$ until it reaches height $n$:
\begin{align}\label{stairtau}
\tau(u,v)=
  \begin{cases}
  i & \ \text{if $u=-L-1$ and $a_i \leq v<a_{i+1}$ or $u=L+1$ and $b_i \leq v<b_{i+1}$, }\\
 0 & \  \text{if $u\in[-L,L]$ and $v=-(M+1)$}\\
n & \ \text{if $u\in[-L,L]$ and $v=M+1$},
  \end{cases}
\end{align}
where $i\in\{0,\dots,n\}$, see Figure \ref{fig:1}.
Note that if two or more values of the $a_i$ or $b_i$ coincide, then the boundary height $\t$ takes jumps
higher than $1$ at those points.

Next, let $Z(a_1,\dots,a_n; b_1,\dots, b_n; L,M):=Z_\L^\t$ denote the partition
function  in $\L=\L_{L,M}$ with the boundary condition $\t$
defined  in \eqref{stairtau} and let as usual $Z_\L$ denote the partition function on $\L$ with zero boundary condition. 
\begin{figure}[htb]
        \centering
 \begin{overpic}[scale=0.27]{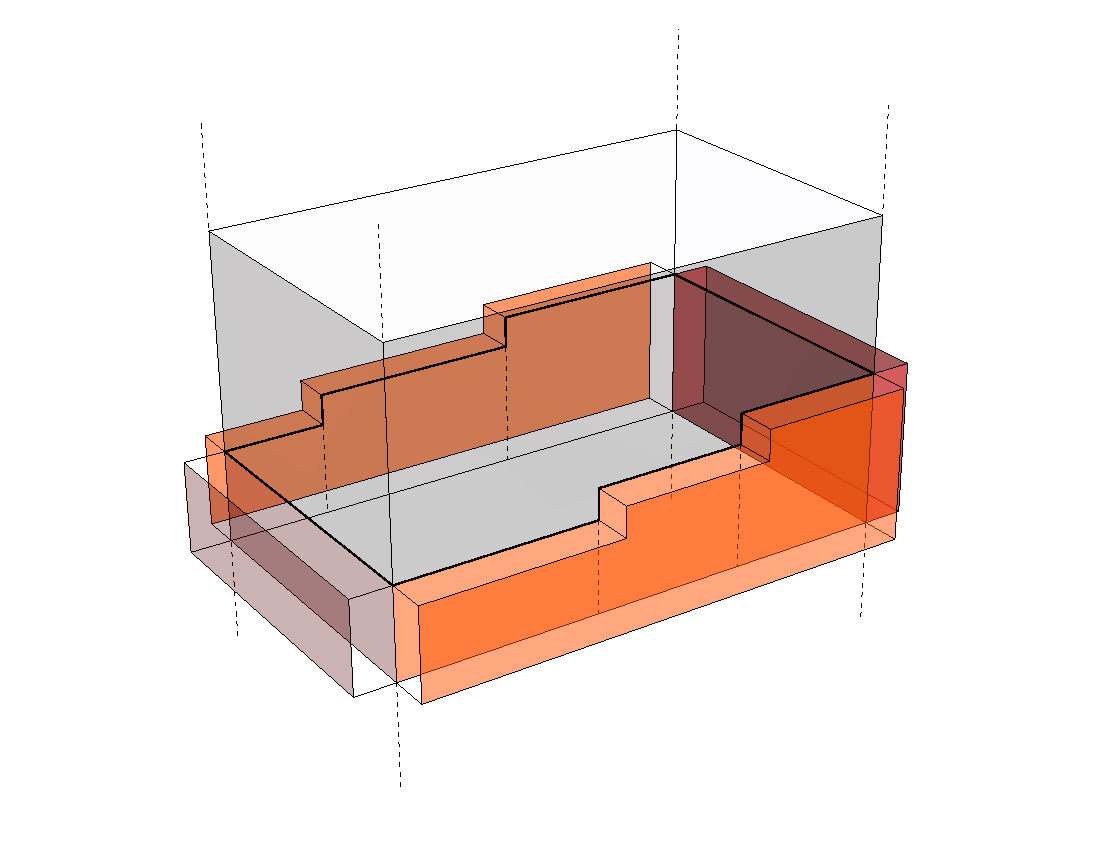}
\put(30.5,29){$z_1$}\put(30.5,31){$\cdot$}
\put(47,33.5){$z_2$}\put(47,35.8){$\cdot$}
\put(56,17){$z'_1$}
\put(68,21){$z'_2$}
\put(68.5,23.7){$\cdot$}
\put(56,19.5){$\cdot$}
\end{overpic}
        \caption{Sketch of the staircase boundary condition (\ref{stairtau}) in the rectangle $\L_{L,M}$ 
          for $n=2$. The points $z_i,z'_i$ have coordinates $z_i=(-L-1,a_i)$, and  $z'_i=(L+1,b_i)$.}
          \label{fig:1}
\end{figure}
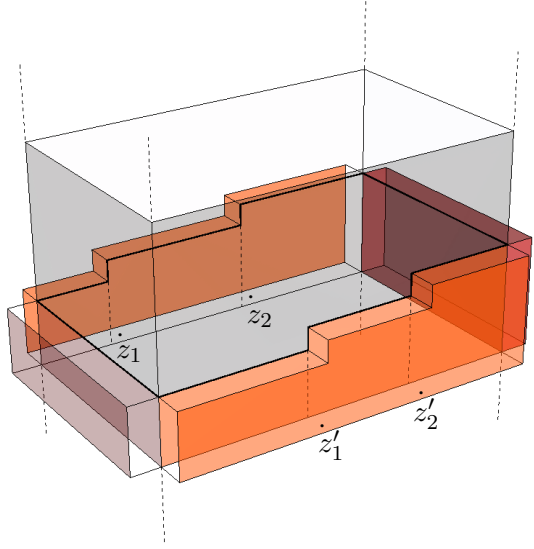
Combining the expansions in Section \ref{clexp} and 
\cite[Section 2.5]{CMT}
one proves that the limit
\begin{equation}
    \label{eq:1}
\cZ(a_1,\dots,a_n;\, b_1,\dots,b_n;\, 
  L) :=    \lim_{M\to \infty}\frac{Z(a_1,\dots,a_n;\, b_1,\dots,b_n;\, 
  L,M)}{Z_\L}
  \end{equation}
  is well defined, for every fixed choice of parameters $a_1\leq\cdots,a_n$ and $b_1\leq \cdots\leq b_n$.

The main result here is that, for all $p\geq 1$, the effective
partition function defined in \eqref{eq:1} has the 
following monotonicity property w.r.t. the parameters
$\{a_i,b_i\}_{i=1}^n$. 
The case $p=1$ is \cite[Theorem 4.1]{CMT}.
\begin{theorem} 
\label{th:mon}
There exists $\b_0>0$ such that, for any $\b>\b_0$, any $a_1\leq\cdots,a_n$ and $b_1\leq \cdots\leq b_n$,  and any $L\in \bbN$
  \begin{gather}
    \label{eq:2}
\cZ(a_1,\dots,a_n;\, b_1,\dots,b_n;\, L)\leq \prod_{i=1}^n \cZ(a_i;\, b_i;\, 
  L).  
  \end{gather}
\end{theorem}
The proof of Theorem \ref{th:mon} is based on the following key lemma. Once Lemma \ref{lem:1} below is established, the theorem follows by a simple iteration as in the proof of  \cite[Theorem 4.1]{CMT}.
\begin{lemma}
\label{lem:1}
For any $a_1\leq\cdots,a_n$ and $b_1\leq \cdots\leq b_n$,  
 \[
\cZ(a_1,\dots,a_n;\ b_1,\dots,b_n;\ 
  L)\leq \cZ(a_1,\dots,a_{n-1},a_n+1;\ b_1,\dots,b_{n-1}, b_n+1;\ 
  L).
\] 
\end{lemma}
\begin{proof}[Proof of  Lemma \ref{lem:1}]
Set $\L:= \L_{L,M}$ for some large fixed $M>\max\{a_n,b_n,-a_1,-b_1\}$. Let
$\t,\t'$ be the boundary conditions associated to
$\{a_i,b_i\}_{i=1}^n$ and $\{a'_i,b'_i\}_{i=1}^n$ according to
\eqref{stairtau}, where $a'_i=a_i,b'_i=b_i$, for $i=1,\dots,n-1$, while $a'_n=a_n+1, b'_n=b_n+1$.  
 It suffices to establish that 
$$\lim_{M\to\infty}\left(\frac{Z_\L^\t}{Z_\L}-\frac{Z_\L^{\t'}}{Z_\L}\right)\leq 0.$$ 
We shall prove the equivalent claim:
 \begin{gather}
    \label{eq:02}
\lim_{M\to\infty}\log\frac{Z_\L^\t}{Z_\L^{\t'}}\leq 0.  
  \end{gather}
Define the points
$z=(-(L+1),a_n),w=(-L,a_n)$ and  $z'=(L+1,b_n),\
w'=(L,b_n)$, so that $w$ (resp.\ $w'$) is the nearest neighbor of $z$ (resp.\ $z'$) in $\L$, see Figure \ref{fig:010}.
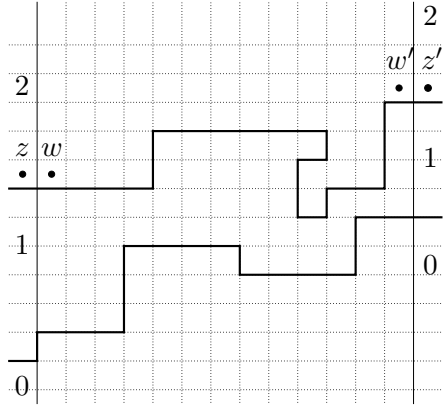
\begin{figure}[htb]
        \centering
 \begin{overpic}[scale=0.3]{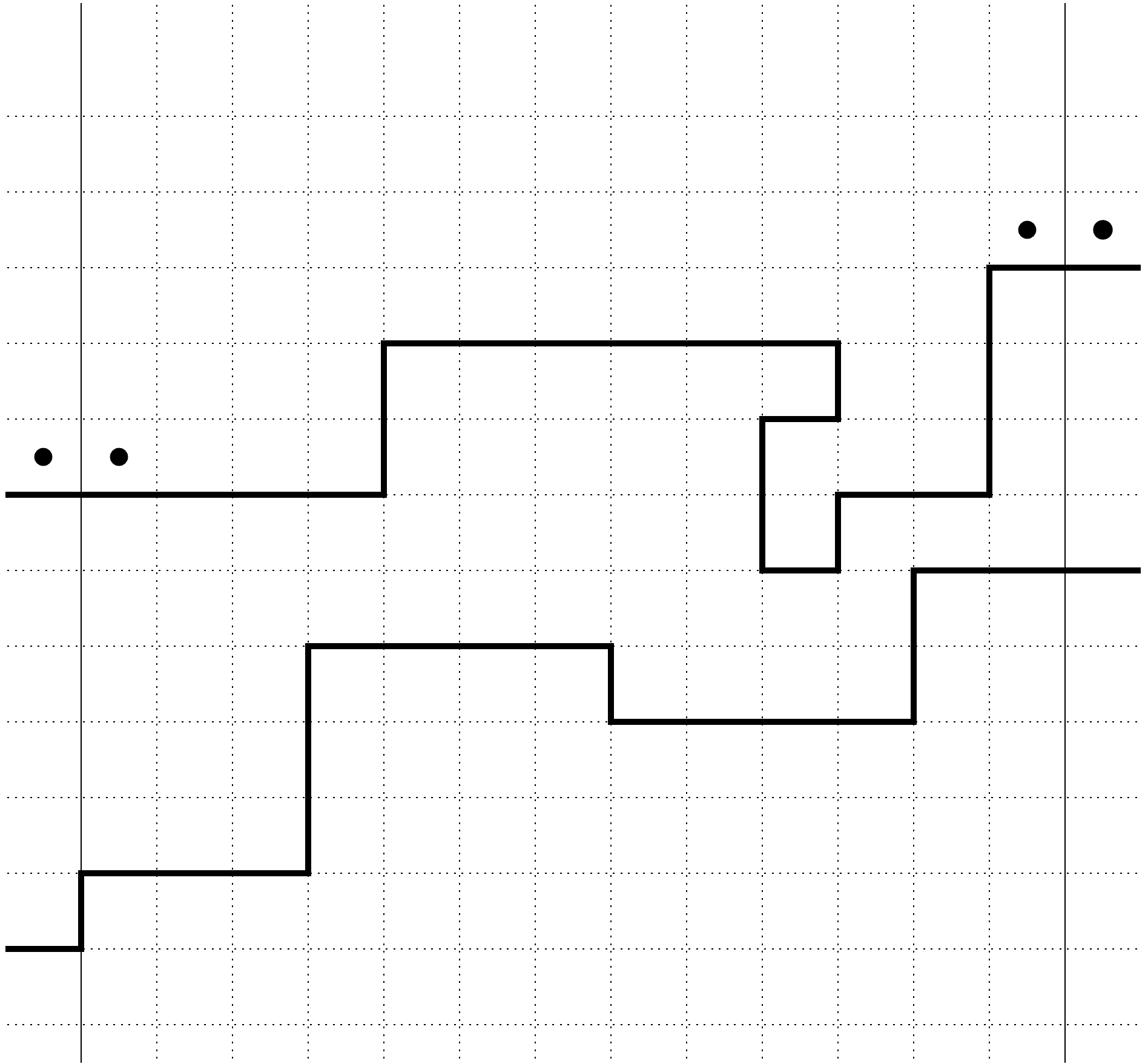}
\put(94.5,77) {$z'$}
\put(2,57){$z$}
\put(8,57){$w$}
\put(86.5,77) {$w'$}
\put(2,2.3){$0$}
\put(95,30){$0$}
\put(2,34.5){$1$}
\put(95,54.5){$1$}
\put(2,71){$2$}
\put(95,87){$2$}
\end{overpic}
        \caption{Sketch of the staircase boundary condition $\t$ with $n=2$ steps as seen from
               above, with two open contours and the pairs of vertices appearing in the proof of  Lemma \ref{lem:1}: 
          $z=(-L-1,a_2), w=(-L,a_2)$, $z'=(L+1,b_2), w'=(L,b_2)$.}\label{fig:010}
\end{figure}

To prove \eqref{eq:02}, let $\t_s=s\t+(1-s)\t'$ 
and write
\begin{equation}\label{diffz}
\log \frac{Z_\L^\t}{Z_\L^{\t'}}= \int_0^1 ds \, \frac{\tfrac{d}{ds}Z_\L^{\tau_s}}{Z_\L^{\tau_s}}
= \int_0^1 ds \, \bbE_\L^{\tau_s}[\psi_s(\phi(w))+\psi_s(\phi({w'}))],
\end{equation}
where \begin{equation}\label{phisz}\psi_s(x):= -\b\,\frac{d}{ds}|x-(n-1+s)|^p,\end{equation}
which is well defined for almost all $s\in[0,1]$, and $\bbE_\L^{\tau_s}$ denotes expectation with respect to $\bbP_\L^{\tau_s}$. 
Observe that $\psi_s(x)$ is increasing in $x$. Indeed, for almost all $s\in[0,1]$
\begin{equation}\label{conv}
\frac{d}{dx}\psi_s(x) = \b\frac{d^2}{dx^2}|x-(n-1+s)|^p\geq 0\,,
\end{equation}
for all $p\geq 1$. Then, by the FKG inequality we can raise the boundary condition to $n-1$ at all boundary vertices which had a boundary condition less than $n-1$. This yields a new boundary condition $\hat\t_s$ 
 such that
$$
\bbE_\L^{\tau_s}[\psi_s(\phi(w))+\psi_s(\phi({w'}))]\leq \bbE_\L^{\hat \tau_s}[\psi_s(\phi(w))+\psi_s(\phi({w'}))].
$$
Thus, from \eqref{diffz}, it suffices to show that 
  \begin{gather}
    \label{eq:22}
\int_0^1 \bbE_{L,\infty}^{\hat \tau_s}[\psi_s(\phi(w))+\psi_s(\phi({w'}))]\,ds  =0,
  \end{gather}
  where $\bbE_{L,\infty}^{\hat \tau_s}$ denotes the expectation in the infinite strip $\L_{L,\infty}$, obtained by taking the limit as $M\to\infty$ of $\bbE_\L^{\hat \tau_s}$. Note that this limiting Gibbs measure is well defined by the usual arguments; see e.g.\ \cite[Section 2.5]{CMT}. 
  
  We are going to show that for all $s\in[0,1]$
  \begin{gather}
    \label{eq:32}
\bbE_{L,\infty}^{\hat \tau_s}[\psi_s(\phi(w))+\psi_s(\phi({w'}))] + \bbE_{L,\infty}^{\hat \tau_{1-s}}[\psi_{1-s}(\phi(w))+\psi_{1-s}(\phi({w'}))]  =0.
  \end{gather}
  Clearly, \eqref{eq:32} implies \eqref{eq:22}.
  
  First, by vertical translation invariance, one can pretend that $n=1$, so that 
  $\hat\t_s$ is $0$ or $1$, except at $z$ and $z'$ where it equals $s$, and that 
  $|x-s|^p$ replaces $|x-(n-1+s)|^p$ in \eqref{phisz} above. In particular, $\psi_s(x)=-\psi_{-s}(-x)$. Then, by symmetry one has 
  $$
  \bbE_{L,\infty}^{\hat \tau_s}[\psi_s(\phi(w))+\psi_s(\phi({w'}))] + \bbE_{L,\infty}^{-\hat \tau_{s}}[\psi_{-s}(\phi(w))+\psi_{-s}(\phi({w'}))]  =0.
  $$
  Finally, by shifting vertically all heights by $+1$ and by a 180 degrees rotation of the $\L_{L,\infty}$ geometry , one sees that
  $$
  \bbE_{L,\infty}^{-\hat \tau_{s}}[\psi_{-s}(\phi(w))+\psi_{-s}(\phi({w'}))]  =\bbE_{L,\infty}^{\hat \tau_{1-s}}[\psi_{1-s}(\phi(w))+\psi_{1-s}(\phi({w'}))].
  $$ 
  This proves \eqref{eq:32}. 
\end{proof}  
 \bigskip
%
 \section{Upper bound}\label{sec:upper}
From \eqref{pospo1}-\eqref{posp1}, we see that once we prove the upper bound for $\bbP(\phi_{\L_L}\geq 0)$, then the same upper bound follows for the finite volume probabilities $
\bbP_{\L_L}(\phi_{\L_L}\geq 0)$. 

We start by summarizing a key result from \cite{LMS}.  The following facts follow from \cite[Theorem 2, Theorem 4 and Section 4.4]{LMS}. The corresponding statements for $p=1$ are in \cite[Theorem 2]{CLMST2}. Below, $\bbP_{\L_L}^+$ denotes the measure $\bbP_{\L_L}$ conditioned on the event $\phi_{\L_L}\geq 0$; see \eqref{posev}.
\begin{proposition}
  \label{mainthm:floor-shape}
  For any $\d>0$ and $K>0$, define $A_L(\d,K)$, as the event that there exists a lattice circuit $\cC$ surrounding $\L':=\L_{(1-\d)L}$
such that $\phi(x)\geq H_p(L) - K$, for all $x\in\cC$, where $H_p(L)$ is defined as in \eqref{eq-H-def}.
There exists $\b_0>0$ such that, for any $p\geq1$, $\beta > \b_0$ one has: For any $\d>0$, there exists $K\in\bbN$ such that
\begin{align}\label{thlms}
\lim_{L\to\infty}\bbP_{\L_L}^+(A_L(\d,K))=1
\end{align}
\end{proposition}
From Proposition \ref{mainthm:floor-shape} and the simple monotonicity argument in \cite[Proposition 5.1]{CMT} we obtain that: for any $\d>0$, there exists $K\in\bbN$ such that 
\begin{align}\label{tho}
\lim_{L\to\infty}\bbP(A_L(\d,K)\tc \phi_{\L_L}\geq 0)=1.
\end{align}
From the straightforward fact that for any event $A$ one has 
$$\bbP(\phi_{\L_L}\geq 0) \leq \frac{\bbP(A)}{
\bbP(A\tc \phi_{\L_L}\geq 0)}\,,$$
it follows that the desired upper bound for $\bbP(\phi_{\L_L}\geq 0) $ is a consequence of the statement: for any $\d>0$, for any $K\in\bbN$,  
\begin{align}\label{tholms}
\limsup_{L\to\infty}\frac1{8LH_p(L)}\log \bbP(A_L(\d,K))\leq - \b\t_{p,\b}(1-\d).
\end{align}
Observe that the event $A_L(\d,K)$ implies that there exist contours $\g_j$, $j=1,\dots, H_p(L)-K,$ such that $$
\L_L\supset\L_{\g_1} \supset\cdots\supset\L_{\g_{H_p(L)-K}}\supset \L_{(1-\d)L},$$
and such that $\g_j$ is a $j$-contour for the surface for each $j$.  
Then the proof of \eqref{tholms} can be obtained by the very same argument used in \cite[Proof of Proposition 5.2]{CMT}. To extend this argument to the setting $p>1$ one only needs to make sure that: 1) the cluster expansion techniques used there apply here as well, and 2)  the crucial monotonicity of partition functions in the staircase ensemble holds for all $p\geq 1$. The first point can be checked with minor modifications of the arguments in Section \ref{clexp}. The second point has been established in Theorem \ref{th:mon} above.  

\section*{Acknowledgments} F. T. was partially supported by the CNRS PICS grant ``Discrete random interfaces and Glauber dynamics''
\bibliographystyle{plain}
\bibliography{sos.bib}

\end{document}